\documentclass[english,final]{amsart}
\usepackage[T1]{fontenc}
\usepackage[latin1]{inputenc}
\usepackage[english]{babel}
\usepackage{amsmath,amsfonts,amssymb,amsthm,enumitem}
\usepackage{mathtools}
\usepackage{mathrsfs}
\usepackage{graphicx}

\usepackage{bbold}

\newcommand*{\MRref}[2]{\linebreak[0] \href{http://www.ams.org/mathscinet-getitem?mr=#1}{MR \textbf{#1}}}

\usepackage[pdftitle={The K-theoretical range of Cuntz-Krieger algebras},
  pdfauthor={Arklint, Bentmann, Katsura},
]{hyperref}
\usepackage[lite]{amsrefs}
\usepackage{microtype}

\usepackage{soul}

\usepackage{xcolor}

\numberwithin{equation}{section}

\theoremstyle{plain}
\newtheorem{theorem}[equation]{Theorem}
\newtheorem{proposition}[equation]{Proposition}
\newtheorem{lemma}[equation]{Lemma}
\newtheorem{corollary}[equation]{Corollary}

\theoremstyle{definition}
\newtheorem{definition}[equation]{Definition}

\theoremstyle{remark}
\newtheorem{remark}[equation]{Remark}

\renewcommand{\phi}{\varphi}
\renewcommand{\epsilon}{\varepsilon}
\newcommand{\Z}{\mathbb Z}

\newcommand{\csa}{\texorpdfstring{$C^*$\nb-al\-ge\-bra}{C*-algebra}}
\newcommand{\csas}{$C^*$\nb-al\-ge\-bras}

\newcommand{\CK}{Cuntz--Krie\-ger algebra}
\newcommand{\CKs}{Cuntz--Krie\-ger algebras}

\renewcommand{\subset}{\subseteq}

\newcommand*{\K}{\textup{K}}

\newcommand{\I}[1]{\mathbb I_{#1}}

\DeclareMathOperator{\im}{im}

\DeclareMathOperator{\coker}{coker}

\newcommand{\Op}{\mathbb O}

\newcommand{\LC}{\mathbb{LC}}

\DeclareMathOperator{\FK}{FK}

\newcommand{\Mod}{\mathfrak{Mod}}
\DeclareMathOperator{\Prim}{Prim}

\newcommand{\onto}{\twoheadrightarrow}
\newcommand{\into}{\hookrightarrow}

\DeclareMathOperator{\clo}{CP}
\newcommand{\ob}{\widetilde\partial{}}

\newcommand{\Ab}{\mathfrak{Ab}}

\DeclareMathOperator{\cok}{coker}
\DeclareMathOperator{\rank}{rank}
\newcommand{\unit}{\textnormal{unit}}
\newcommand{\pt}{\textnormal{pt}}

\newcommand{\Catgunnar}{\mathcal{R}}

\DeclareMathOperator{\Path}{Path}
\DeclareMathOperator{\DoublePaths}{DP}

\newcommand{\FKgunnar}{\FK_\Catgunnar}

\newcommand{\obd}[1]{\widetilde{\partial}\{{#1}\}}

\newcommand{\osi}[1]{\widetilde{\{{#1}\}}}

\newcommand{\si}[1]{\{{#1}\}}

\newcommand{\gsi}[1]{{#1}_{1}}
\newcommand{\gobd}[1]{\widetilde\partial{#1}_{0}}
\newcommand{\gosi}[1]{\widetilde{{#1}}_0}

\newcommand{\rsi}[1]{{#1}_0}

\newcommand*{\nb}{\nobreakdash}
\newcommand*{\Cstar}{\texorpdfstring{$\textup C^*$\nb-}{C*-}}
\newcommand*{\Star}{\texorpdfstring{$^*$\nobreakdash-}{*-}}
\newcommand{\shom}{$^*$\nobreakdash-homo\-mor\-phism}
\newcommand*{\cstaralg}{\mathfrak{C^*alg}}

\newcommand*{\defeq}{\mathrel{\vcentcolon=}}

\input xy
\xyoption{all}
\newdir{ >}{{}*!/-5pt/@{>}}

\title{The $\K$\nb-theoretical range of Cuntz--Krie\-ger algebras}

\author{Sara E. Arklint}
\address{Department of Mathematical Sciences, University of Copenhagen, Uni\-versi\-tets\-park\-en~5, DK-2100 Copenhagen, Denmark}
\email{arklint@math.ku.dk}

\author{Rasmus Bentmann}
\address{Mathematisches Institut\\
  Georg-August Universit\"at G\"ottingen\\
  Bunsenstra\ss{}e 3--5\\
  37073 G\"ottingen\\
  Germany}
\email{rbentma@uni-math.gwdg.de}

\author{Takeshi Katsura}
\address{Department of Mathematics, Keio University, 3-14-1 Hiyoshi, Kouhoku-ku, Yokohama 223-8522, Japan}
\email{katsura@math.keio.ac.jp }

\thanks{This research was supported by the Danish National Research Foundation through the Centre for Symmetry and Deformation (DNRF92). The third-named author was partially supported by the Japan Society for the Promotion of Science.}
\keywords{Cuntz--Krie\-ger algebras, classification, filtered $\K$\nb-theory}
\subjclass[2010]{46L35, 46L80, (46L55)}

\begin{document}
\bibliographystyle{alpha}

\begin{abstract}
We augment Restorff's classification of purely infinite \CKs{} by describing the range of his invariant on purely infinite \CKs{}. We also describe its range on purely infinite graph \csas{} with finitely many ideals, and provide `unital' range results for purely infinite \CKs{} and unital purely infinite graph \csas{}.
\end{abstract}

\maketitle

\section{Introduction}

\CKs{} form a class of \csas{} closely related to symbolic dynamics~\cites{cuntz_krieger,cuntz}. Based on this relationship, classification results for purely infinite \CKs{} by $\K$\nb-theoretical invariants have been established by Mikael R{\o{}}rdam in the simple case~\cite{Rordam:CK} and by Gunnar Restorff in the case of finitely many ideals~\cite{restorff}.

For simple \CKs{}, the $\K_0$-group suffices for classification (because the $\K_1$-group can be identified with the free part of the $\K_0$-group). Moreover, it is known that every finitely generated abelian group arises as the $\K_0$-group of some simple \CK{}~\cite{ektw}*{Proposition~3.9}.

The invariant in Restorff's classification theorem for non-simple purely infinite \CKs{} is called \emph{reduced filtered $\K$\nb-theory}; we denote it by~$\FKgunnar$. Being an almost precise analogue of the $\K$\nb-web of Boyle and Huang~\cite{boyle_huang} in the world of \csas{}, it comprises the $\K_0$-groups of certain distinguished ideals and the $\K_1$-groups of all simple subquotients, along with the action of certain natural maps.

The first aim of this article is to clarify the definition of the target category of reduced filtered $\K$\nb-theory. We define a certain pre-additive category~$\Catgunnar$ such that~$\FKgunnar$ becomes a functor to $\Catgunnar$\nb-mod\-ules in a natural way. Our second aim is then to determine the class of $\Catgunnar$\nb-mod\-ules that arise (up to isomorphism) as the reduced filtered $\K$\nb-theory of some (tight, purely infinite) Cuntz--Krie\-ger algebra. This involves a natural exactness condition, as well as some conditions that translate well-known $\K$\nb-theoretical properties of purely infinite \CKs{}.

A \CK{} is purely infinite if and only if it has finitely many ideals, and if and only if it has real rank zero~\cite{hongszymanski}. For a \csa{} with real rank zero, the exponential map in the $\K$\nb-theoretical six-term exact sequence for every inclusion of subquotients vanishes~\cite{brownpedersen}. This fact is crucial to our definitions and results, in this article and in the companion article~\cite{reduction}.

Our work is based on the article~\cite{ektw} of Eilers, Katsura, Tomforde and West who characterized the six-term exact sequences in $\K$\nb-theory of \CKs{} with a unique non-trivial proper ideal. We obtain our result by a careful inductive application of the result in~\cite{ektw}. 

Combining our range result for purely infinite \CKs{} with Restorff's classification theorem, we obtain an explicit natural description of the set of stable isomorphism classes of purely infinite \CKs{}, completing the picture in a way previously known only in the simple case and the one-ideal case.

Anticipating potential future classification results generalizing Restorff's theorem, we also provide a range result for purely infinite graph \csas{} with finitely many ideals.
\CK{}s may be viewed as a specific type of graph \csa{}s, namely those arising from finite graphs with no sources~\cite{arklint_ruiz}. Indeed, our range result is established by graph \csa{}ic methods.  All purely infinite graph \csa{}s have real rank zero~\cite{hongszymanski}, so that the abovementioned $\K$\nb-theoretical particularities that make our approach work are still present in this more general setting.

Finally, we equip reduced filtered $\K$\nb-theory with a unit class and establish corresponding range results for purely infinite \CKs{} and unital purely infinite graph \csas{}. In~\cite{reduction}, this is used to give an ``external characterization'' of purely infinite \CK{}s under some conditions on the ideal structure.

\subsection{Acknowledgements}
Most of this work was done while the third-named author stayed at the University of Copenhagen. He would like to thank the people in Copenhagen for their hospitality. The authors are grateful to S{\o{}}ren Eilers for his encouragement and valuable comments, to Gunnar Restorff for suggesting improvements in exposition on an earlier version of the article, and to Efren Ruiz and the anonymous referee for a number of corrections.

\section{Preliminaries} \label{sec:notation}

In this article, matrices act from the right and the composite of maps $A\xrightarrow{f} B\xrightarrow{g} C$ is denoted by~$fg$.
The category of abelian groups is denoted by $\Ab$, the category of $\Z/2$-graded abelian groups by $\Ab^{\mathbb Z/2}$.
We let $\Z_+$ denote the set of non-negative integers. When $S$ is a set, we use the symbol $M_S$ to indicate the set of square matrices whose rows and columns are indexed by elements in~$S$.

\subsection{Finite spaces} \label{sec:finite_spaces}
Throughout the article, let $X$ be a finite $T_0$-space, that is, a finite topological space in which no two different points have the same open neighbourhood filter.
For a subset $Y$ of $X$, 
we let $\overline Y$ denote the closure of $Y$ in $X$, and 
we let $\overline\partial Y$ denote the boundary 
$\overline Y \setminus Y$ of $Y$. 
Since $X$ is a finite space, 
there exists a smallest open subset $\widetilde Y$ of $X$ containing $Y$.
We let $\ob Y$ denote the set $\widetilde Y\setminus Y$.

For $x,y\in X$ we write $x\leq y$ 
when $\overline{\{x\}} \subset \overline{\{y\}}$, 
and $x < y$ when $x\leq y$ and $x \neq y$. 
We write $y\to x$ when $x < y$ and 
no $z \in X$ satisfies $x < z < y$.
The following lemma is straightforward to verify.

\begin{lemma}\label{Lem:on_previous}
For an element $x \in X$, the following hold:
\begin{enumerate}[label=\textup{(\arabic*)}]
\item An element $y \in X$ satisfies $y\to x$ if and only if 
$y$ is a closed point of $\ob \{x\}$.
\item We have $\ob \{x\} = \displaystyle\bigcup_{y\to x} \widetilde{\{y\}}$, 
and consequently $\ob \{x\}$ is open. 
\item An element $y \in X$ satisfies $x \leq y$ if and only if 
there exists a finite sequence $(z_k)_{k=1}^n$ in $X$ 
such that $z_{k+1}\to z_{k}$ for $k=1, \ldots, n-1$ 
where $z_1=x$, $z_{n}=y$. 
\end{enumerate}
\end{lemma}

We call a sequence $(z_k)_{k=1}^n$ 
as in Lemma~\ref{Lem:on_previous}(3) a \emph{path} from $y$ to $x$. 
We denote by $\Path(y,x)$ the set of paths from $y$ to $x$.
Thus Lemma~\ref{Lem:on_previous}(3) can be rephrased as follows: 
two points $x,y \in X$ satisfy $x \leq y$ if and only if 
there exists a path from $y$ to $x$. 
Such a path is not unique in general.
Two points $x,y \in X$ satisfy $y\to x$ 
if and only if $(x,y)$ is a path from $y$ to $x$;
in this case, there are no other paths from $y$ to $x$.

\subsection{\Cstar{}algebras over finite spaces}
Recall from~\cite{MN:Bootstrap}, that a \emph{\csa{} $A$ over~$X$} is a \csa{} $A$ equipped with a continuous map $\Prim(A)\to X$ or, equivalently, an infima- and suprema-preserving map $\Op(X)\to\I{}(A), U\mapsto A(U)$ mapping open subsets in~$X$ to (closed, two-sided) ideals in~$A$ (in particular, one has $A(\emptyset)=0$ and $A(X)=A$).
The \csa{} $A$ is called \emph{tight} over $X$ if this map is a lattice isomorphism.
A \shom{} $\phi\colon A\to B$ for \csa s $A$ and $B$ over $X$ is called \emph{$X$-equivariant} if $\phi\bigl(A(U)\bigr)\subseteq B(U)$ for all $U\in\Op(X)$.
The category of \csas{} over~$X$ and $X$\nb-equivariant \shom{}s is denoted by $\cstaralg(X)$.

Let $\LC(X)$ denote the set of locally closed subsets of $X$, that is, subsets of the form $U\setminus V$ with $U$ and $V$ open subsets of $X$ satisfying $V\subseteq U$.
For $Y\in\LC(X)$, and $U,V\in\Op(X)$ satisfying that $Y=U\setminus V$ and $U\supseteq V$, we define $A(Y)$ as $A(Y)=A(U)/A(V)$, which up to natural isomorphism is independent of the choice of $U$ and $V$ (see \cite{MN:Bootstrap}*{Lemma 2.15}).
For a \csa{} $A$ over $X$, 
the $\Z/2$-graded abelian group $\FK_Y^*(A)$ 
is defined as $\K_*\bigl(A(Y)\bigr)$ for all $Y\in\LC(X)$.  
Thus $\FK_Y^*$ is a functor from $\cstaralg(X)$
to the category $\Ab^{\mathbb Z/2}$ of $\Z/2$-graded abelian groups;
compare \cite{MN:Filtrated}*{\S2}.

\begin{definition}  \label{def:generators}
Let $Y\in\LC(X)$, $U\subset Y$ be open and set $C=Y\setminus U$. 
A pair $(U,C)$ obtained in this way is called a \emph{boundary pair}.
The natural transformations occuring in the six-term exact sequence in $\K$\nb-theory
for the distinguished subquotient inclusion  associated to $U\subset Y$ are denoted by $i_U^Y$, $r_Y^C$ and $\delta_C^Y$:
\[
\xymatrix{
\FK_U^* \ar[rr]^-{i_U^Y} && \FK_Y^*.\ar[dl]^-{r_Y^C} \\
& \FK_C^*\ar[ul]|\circ^-{\delta_C^U} &
}
\]
\end{definition}

\begin{lemma}
  \label{lem:relations_for_gunnar}
Let $(U,C)$ be a boundary pair and let $V\subseteq U$ be an open subset. The following relations hold: 
\begin{enumerate}[label=\textup{(\arabic*)}]
\item $\delta_C^U i_U^Y = 0$;
\item $i_V^U i_U^Y = i_V^Y$.
\end{enumerate}
\end{lemma}

\begin{proof}
The first statement follows from the exactness of the six-term sequence in $\K$\nb-theory. The second statement already holds for the ideal inclusions inducing the relevant maps on $\K$\nb-theory.
\end{proof}

\subsection{Graph \csa{}s}
We follow the notation and definition for graph \csas{} of Iain Raeburn's monograph~\cite{raeburn}; this is also our reference for basic facts about graph \csas{}.
All graphs are assumed to be countable and to satisfy Condition (K), hence all considered graph \csas{} are separable and of real rank zero \cite{hongszymanski}*{Theorem~2.5}.

\begin{definition}[] \label{def:graph}
Let $E = (E^0,E^1,s,r)$ be a countable directed graph.
The \emph{graph \csa{}} $C^*(E)$ is defined as the universal \csa{} generated by
a set of mutually orthogonal projections $\{ p_v \mid v \in E^0 \}$ and a set $\{ s_e \mid e \in E^1 \}$ of partial isometries satisfying the relations

\begin{itemize}
	\item $s_e^* s_f = 0$ if $e,f \in E^1$ and $e \neq f$,
	\item $s_e^* s_e = p_{s(e)}$ for all $e \in E^1$,
	\item $s_e s_e^* \leq p_{r(e)}$ for all $e \in E^1$, and,
	\item $p_v = \sum_{e \in r^{-1}(v)} s_e s_e^*$ for all $v \in E^0$ with $0 < |r^{-1}(v)| < \infty$.
\end{itemize}
\end{definition}

The graph \csa{} $C^*(E)$ is a \CK{} if and only if the graph~$E$ is finite with no sources, see~\cite{arklint_ruiz}.

\begin{definition}
Let $E$ be a directed graph.
An edge $e\in E^1$ in $E$ is called a \emph{loop} if $s(e)=r(e)$.
A vertex $v\in E^0$ in $E$ is called \emph{regular} if $r^{-1}(v)$ is finite and nonempty.
\end{definition}

If all vertices in $E$ support two loops, then $C^*(E)$ is purely infinite, see~\cite{hongszymanski}*{Theorem~2.3}.

\begin{definition}
Let $E$ be a directed graph.  
A \emph{path of length $n$} in $E$ is a finite sequence $e_1\cdots e_n$ of edges $e_i\in E^1$ with $s(e_i)=r(e_{i+1})$ for all $i$. Vertices in $E$ are regarded as \emph{paths of length $0$}.
A path $e_1\cdots e_n$ is called a \emph{return path} if $r(e_i)\neq r(e_1)$ for all $i\neq 1$.

For vertices $v,w$ in $E$, we write $v\geq w$ if there is a path in $E$ from $v$ to $w$, i.e., a path $e_1\cdots e_n$ with $s(e_n)=v$ and $r(e_1)=w$. In particular, $v\geq v$ for all $v\in E^0$.
\end{definition}

\begin{definition}
Let $E$ be a directed graph.
Let $H$ be a subset of $E^{0}$.  The subset $H$ is called \emph{saturated} if $s(r^{-1}(v))\subseteq H$ implies $v\in H$ for all regular vertices $v$ in $E$.   When $H$ is saturated,  we let $I_H$ denote the ideal in $C^*(E)$ generated by $\{ p_{v} \mid v \in H \}$.
The subset $H$  is called \emph{hereditary} if for all $w\in H$ and $v\in E^0$, $v\geq w$ implies $v\in H$. 

A vertex $v\in E^0$ in $E$ is called a \emph{breaking vertex} with respect to the saturated hereditary subset $H$ if $s(r^{-1}(v))\cap H$ is infinite and $s(r^{-1}(v))\setminus H$ is finite and nonempty.
\end{definition}

\begin{definition}
Let $E$ be a directed graph.  
We say that $E$ satisfies \emph{Condition \textup{(}K\textup{)}} if for all edges $v\in E^0$ in $E$, either there is no path of positive length in $E$ from~$v$ to~$v$ or there are at least two distinct return paths of positive length in $E$ from~$v$ to~$v$.  We call $E$ \emph{row-finite} when $r^{-1}(v)$ is finite for all $v\in E^0$.
\end{definition}

When $E$ satisfies Condition \textup{(}K\textup{)} and no saturated hereditary subsets in $E^0$ have breaking vertices, then the map $H\mapsto I_H$ defines a lattice isomorphism between the saturated hereditary subsets in $E^0$ and the ideals in $C^*(E)$, see~\cite{drinentomforde}*{Theorem~3.5}.

\section{Reduced filtered \texorpdfstring{$\K$}{K}-theory} \label{gunnar}

In this section we introduce the functor $\FKgunnar$ which (for real-rank-zero \csa s) is equivalent to the reduced filtered $\K$\nb-theory defined by Gunnar Restorff in~\cite{restorff}.

\begin{definition}
Let $\Catgunnar$ denote the universal pre-additive category generated by objects $\gsi x, \gobd x, \gosi x $ for all $x\in X$ and morphisms $\delta_{\gsi x}^{\gobd x}$ and $i_{\gobd x}^{\gosi x}$ for all $x\in X$, and $i_{\gosi y}^{\gobd x}$ when $y\to x$, subject to the relations
  \label{def:Catgunnar}
\begin{equation}
  \label{eq:gunnarrelation1}
\delta_{\gsi x}^{\gobd x} i_{\gobd x}^{\gosi x} =0
\end{equation}
\begin{equation}
  \label{eq:gunnarrelation2}
 i_p i_{\gosi{y(p)}}^{\gobd x} =  i_q i_{\gosi{y(q)}}^{\gobd x}
\end{equation}
for all $x\in X$, all $y\in X$ satisfying $y> x$, and all paths $p,q\in\Path(y,x)$, where for a path $p=(z_k)_{k=1}^n$ in $\Path(y,x)$, we define $y(p)=z_2$, and
\[
i_p = i_{\gosi{z_n}}^{\gobd{z_{n-1}}} i_{\gobd{z_{n-1}}}^{\gosi{z_{n-1}}}\cdots i_{\gosi{z_3}}^{\gobd{z_2}} i_{\gobd{z_2}}^{\gosi{z_2}}.
\]
Here subscripts indicate domains of morphisms and superscripts indicate co\-domains.
\end{definition}

\begin{definition}[\emph{Reduced filtered $\K$\nb-theory}]
The functor \[\FKgunnar\colon\cstaralg(X)\to\Mod(\Catgunnar)\] is defined as follows:
For a \csa{} $A$ over $X$ and $x\in X$, we define
\begin{align*}
\FKgunnar(A)(\gsi x) &= \FK_{\si x}^1(A)\\
\FKgunnar(A)(\gobd x) &= \FK_{\obd x}^0(A)\\
\FKgunnar(A)(\gosi x) &= \FK_{\osi x}^0(A)
\end{align*}
using the notation from \S\ref{sec:finite_spaces}, and for a morphism~$\eta$ in~$\Catgunnar$, we set $\FKgunnar(A)(\eta)$ to be the corresponding map constructed in Definition~\ref{def:generators}.
On an $X$-equivariant \Star{}homomorphism, $\FKgunnar$ acts in the obvious way dictated by its entry functors $\FK_Y$.
It follows from Lemma \ref{lem:relations_for_gunnar} that the functor~$\FKgunnar$ indeed takes values in $\Catgunnar$\nb-mod\-ules.
\end{definition}

\begin{remark}
We would like to make the reader aware of the following slight subtlety. It may happen that $\obd x = \osi{y}$ for two points $x,y\in X$. But, if $M$ is an exact $\Catgunnar$\nb-mod\-ule in the sense of the definition below, then the map $i_{\gosi y}^{\gobd x}\colon M(\gosi y)\to M(\gobd x)$ is an isomorphism. More generally, if $\obd x$ decomposes as a disjoint union $\bigsqcup_{i=1}^n \osi{y_i}$, then there is an isomorphism $\bigoplus_{i=1}^n M(\gosi{y_i})\to M(\gobd x)$.
\end{remark}

\begin{definition}
For an element $x$ in $X$, let $\DoublePaths(x)$ denote the set of pairs of distinct paths $(p,q)$ in $X$ to $x$ and from some common element which is denoted by $s(p,q)$.
An $\Catgunnar$\nb-mod\-ule $M$ is called \emph{exact} if the sequences
\begin{equation}
  \label{eq:Catgunnarrel1}
M(\gsi x) \xrightarrow{\delta_{\gsi x}^{\gobd x}} M(\gobd x) \xrightarrow{i_{\gobd x}^{\gosi x}} M(\gosi x)
\end{equation}
\begin{equation}
  \label{eq:Catgunnarrel2}
\bigoplus_{(p,q)\in \DoublePaths(x)} M(\gosi{s(p,q)}) \xrightarrow{(i_p-i_q)} \bigoplus_{y\to x} M(\gosi y) \xrightarrow{(i_{\gosi y}^{\gobd x})} M(\gobd x) \longrightarrow 0
\end{equation}
are exact for all $x\in X$.
\end{definition}

\begin{lemma} \label{lem:exactseq}
Let $A$ be a \csa{} over~$X$ with real rank zero.
Let $Y$ be an open subset of $X$ and let $(U_i)_{i\in I}$ be an open covering of $Y$ satisfying $U_i\subseteq Y$ for all $i\in I$. Then the following sequence is exact:
\[
\bigoplus_{i,j\in I} \FK^0_{U_i\cap U_j}(A)\xrightarrow{(i_{U_i\cap U_j}^{U_i}-i_{U_i\cap U_j}^{U_j})}
\bigoplus_{i\in I} \FK^0_{U_i}(A)\xrightarrow{(i_{U_i}^Y)}
\FK^0_Y(A) \longrightarrow 0.
\]
\end{lemma}

\begin{proof}
Using an inductive argument as in \cite{Bredon}*{Proposition 1.3}, we can reduce to the case that $I$ has only two elements.
In this case, exactness follows from a straightforward diagram chase using the exact six-term sequences of the involved ideal inclusions. Here we use that the exponential map $\FK_{V\setminus U}^0(A)\to\FK_U^1(A)$ vanishes for every closed subset~$U$ of a locally closed subset~$V$ of~$X$ if $A$ has real rank zero \cite{brownpedersen}*{Theorem~3.14}.
\end{proof}

\begin{corollary} \label{cor:FKgunnarexact}
Let $A$ be a \csa{} over~$X$ with real rank zero. Then $\FKgunnar(A)$ is an exact $\Catgunnar$\nb-mod\-ule.
\end{corollary}

\begin{proof}
We verify the exactness of the desired sequences in $\FKgunnar(A)$.
The sequence \eqref{eq:Catgunnarrel1} is exact since it is part of the six-term sequence associated to the open inclusion $\widetilde{\partial}\{x\}\subset \widetilde{\{x\}}$.
To prove exactness of the sequence \eqref{eq:Catgunnarrel2}, we apply the previous lemma to the covering $(\widetilde{\{y\}})_{y\to x}$ of $Y=\widetilde{\partial}\{x\}$ and get the exact sequence
\begin{multline*}
\bigoplus_{y\to x, y'\to x} \FK_{\widetilde{\{y\}}\cap\widetilde{\{y'\}}}^0(A)
\xrightarrow{\left(i_{\widetilde{\{y\}}\cap\widetilde{\{y'\}}}^{\widetilde{\{y\}}}-i_{\widetilde{\{y\}}\cap\widetilde{\{y'\}}}^{\widetilde{\{y'\}}}\right)}
\bigoplus_{y\to x} \FK_{\widetilde{\{y\}}}^0(A) \\
\xrightarrow{\left(i_{\widetilde{\{y\}}}^{\widetilde\partial\{x\}}\right)}
\FK_{\widetilde{\partial}\{x\}}^0(A) \longrightarrow 0.
\end{multline*}
Another application of the previous lemma shows that $\displaystyle\bigoplus_{(p,q)\in \DoublePaths(x)} \FK_{\widetilde{s(p,q)}}^0(A)$ surjects onto $\displaystyle\bigoplus_{y\to x, y'\to x} \FK_{\widetilde{\{y\}}\cap\widetilde{\{y'\}}}^0(A)$ in a way making the obvious triangle commute. This establishes the exact sequence \eqref{eq:Catgunnarrel2}.
\end{proof}

\section{Range of reduced filtered \texorpdfstring{$\K$}{K}-theory} \label{range}

In this section, we determine the range of reduced filtered $\K$\nb-theory with respect to the class of purely infinite graph \csa{}s and, by specifying appropriate additional conditions, on the subclass of purely infinite \CK{}s. First, we recall relevant definitions and properties of graph \csa{}s, and explain how one can determine, for a graph~$E$, whether the graph \csa{} $C^*(E)$ can be regarded as a (tight) \csa{} over a given finite space $X$.  We also introduce a formula from~\cite{cet} for calculating reduced filtered $\K$\nb-theory of a graph \csa{} using the adjacency matrix of its defining graph. Finally, Proposition~\ref{prop:range} in conjunction with Theorem~\ref{thm:range} consitutes the desired range-of-invariant result.

\subsection{Calculating reduced filtered \texorpdfstring{$\K$\nb-theory}{K-theory} of a graph \csa{}} \label{ssec:calc}
Let $E$ be a countable graph and assume that all vertices in $E$ are regular and support at least two loops.
Then $E$ satisfies Condition \textup{(}K\textup{)} and has no breaking vertices, so since all subsets of $E^0$ are saturated, the map $H\mapsto I_H$ defines a lattice isomorphism from the hereditary subsets of $E^0$ to the ideals of $C^*(E)$.
Given a map $\psi\colon E^0\to X$ satisfying $\psi\bigl(s(e)\bigr)\geq\psi\bigl(r(e)\bigr)$ for all $e\in E^1$, $\psi(v)\geq\psi(w)$ holds for all $v,w\in E^0$ with $v\geq w$.
We may therefore define a structure on $C^*(E)$ as a \csa{} over $X$ by $U\mapsto I_{\psi^{-1}(U)}$ for $U\in\Op(X)$.

Assume that such a map $\psi$ is given, that is, that $C^*(E)$ is a \csa{} over $X$.
Define for each subset $F\subseteq X$ a matrix $D_F\in M_{\psi^{-1}(F)}(\Z_+)$ as $D_F=A_F-1$, where $A_F(v,w)$ is defined for $v,w\in\psi^{-1}(F)$ by
\[ A_F(v,w)= \#\{ e\in E^1 \mid r(e)=v, s(e)=w \},\]
the number of edges in~$E$ from~$w$ to~$v$; here $1$ denotes the identity matrix.  For subsets $S_1,S_2\subseteq F$, we let $D_F|_{S_1}^{S_2}$ denote the $S_1\times S_2$ matrix given by $D_F|_{S_1}^{S_2}(s_1,s_2)=D_F(s_1,s_2)$ for all $s_1\in S_1$ and $s_2\in S_2$.

Note that a given map $\psi\colon E^0\to X$ turns $C^*(E)$ into a \csa{} over~$X$ if and only if
$D_X|_{\psi^{-1}(y)}^{\psi^{-1}(z)}$ vanishes when $y\not\leq z$.
And if furthermore $D_X|_{\psi^{-1}(y)}^{\psi^{-1}(z)}$ is non-zero whenever $y<z$, then $C^*(E)$ is tight over~$X$ (this condition for tightness is sufficient but not necessary).

Let a map $\psi\colon E^0\to X$ satisfying $\psi\bigl(s(e)\bigr)\geq\psi\bigl(r(e)\bigr)$ for all $e\in E^1$ be given.
Then $\FKgunnar\bigl(C^*(E)\bigr)$ can be computed in the following way.
Let $Y\in\LC(X)$ and $U\in\Op(Y)$ be given, and define $C=Y\setminus U$.
Then by~\cite{cet}, the six-term exact sequence induced by $C^*(E)(U) \into C^*(E)(Y) \onto C^*(E)(C)$ is naturally isomorphic to the sequence
\begin{equation} \label{eq:sixtermsnake}
\begin{split}
\xymatrix{
\cok D_U \ar[r] & \cok D_Y \ar[r] & \cok D_C \ar[d]^-{0} \\
\ker D_C \ar[u]^-{\left[D_Y|_{\psi^{-1}(C)}^{\psi^{-1}(U)}\right]} & \ker D_Y \ar[l] & \ker D_U \ar[l]
}\end{split} \end{equation}
induced, via the Snake Lemma, by the commuting diagram
\[ \xymatrix{
\Z^{\psi^{-1}(U)}\ar[r]\ar[d]^-{D_U} & \Z^{\psi^{-1}(Y)}\ar[r]\ar[d]^-{D_Y} & \Z^{\psi^{-1}(C)}\ar[d]^-{D_C} \\
\Z^{\psi^{-1}(U)}\ar[r] & \Z^{\psi^{-1}(Y)}\ar[r] & \Z^{\psi^{-1}(C)} .
} \]
A more general formula is given in~\cite{cet} for the case where $E$ is not row-finite.  This will be needed in~\S\ref{sec:unital}.

When calculating the reduced filtered $\K$\nb-theory of a graph \csa{}, we will denote the maps in the sequence~\eqref{eq:sixtermsnake} by $\iota$, $\pi$, and $\Delta$, indexed as the natural transformations $i$, $r$, and $\delta$ in Definition~\ref{def:generators}.  For a path $p$ in $X$, a composite $\iota_p$ of natural transformations  is defined as in Definition~\ref{def:Catgunnar}.

\subsection{Range of reduced filtered \texorpdfstring{$\K$\nb-theory}{K-theory} for graph \csa{}s}

The following theorem by S\o{}ren Eilers, Mark Tomforde, James West and the third named author, determines the range of filtered $\K$\nb-theory over the two-point space $\{1,2\}$ with $2\to 1$.  To apply it in the proof of Theorem~\ref{thm:range}, we quote it here reformulated for matrices acting from the right (thereby changing column-finiteness to row-finiteness, etc.).

\begin{theorem}[{\cite{ektw}*{Propositions 4.3 and 4.7}}] \label{ektw}
Let
\[
\xymatrix{
G_1\ar[r]^-{\epsilon} & G_2 \ar[r]^-{\gamma} & G_3 \ar[d]^-{0} \\
F_3\ar[u]^-{\delta} & F_2\ar[l]^-{\gamma'} & F_1\ar[l]^-{\epsilon'}
}
\]
be an exact sequence~$\mathcal E$ of abelian groups with $F_1$, $F_2$, $F_3$ free.
Suppose that there exist row-finite matrices $A\in M_{n_1,n_1'}(\Z)$ and  $B\in M_{n_3,n_3'}(\Z)$ for some $n_1,n_1',n_3,n_3'\in\{1,2,\ldots,\infty\}$ with isomorphisms
\[ \alpha_1\colon\coker A\to G_1, \quad \beta_1\colon\ker A\to F_1, \]
\[ \alpha_3\colon\coker B\to G_3, \quad \beta_3\colon\ker B\to F_3. \]
Then there exist a row-finite matrix $Y\in M_{n_3,n_1'}(\Z)$ and isomorphisms
\[ \alpha_2\colon\coker \begin{pmatrix} A & 0 \\ Y & B \end{pmatrix} \to G_2, \quad \beta_2\colon\ker \begin{pmatrix} A & 0 \\ Y & B \end{pmatrix} \to F_2 \]
such that the tuple $(\alpha_1,\alpha_2,\alpha_3,\beta_1,\beta_2,\beta_3)$ gives an isomorphism of complexes from the exact sequence
\[
\xymatrix{ 
\coker A\ar[r]^-{I} & \coker {\begin{pmatrix} A & 0 \\ Y & B \end{pmatrix}} \ar[r]^-{P} & \cok B \ar[d]^-{0} \\
\coker B\ar[u]^-{[Y]} & \coker {\begin{pmatrix} A & 0 \\ Y& B \end{pmatrix}} \ar[l]^-{P'} & \coker A, \ar[l]^-{I'}
}
\]
where the maps $I, I'$ and $P,P'$ are induced by the obvious inclusions and projections,
to the exact sequence~$\mathcal E$.

If there exist an $A'\in M_{n_1',n_1}(\Z)$ such that $A'A-1\in M_{n_1',n_1'}(\Z_+)$, then $Y$ can be chosen such that $Y \in M_{n_3,n_1'}(\Z_+)$.
If furthermore a row-finite matrix $Z\in M_{n_3,n_1'}(\Z)$ is given, then $Y$ can be chosen such that $Y-Z \in M_{n_3,n_1'}(\Z_+)$.
\end{theorem}

\begin{proposition} \label{prop:range}
Let $A$ be a purely infinite graph \csa{} over $X$.  Then $\FKgunnar(A)$ is an exact $\Catgunnar$\nb-mod\-ule, and $\FK_{\si x}^1(A)$ is free for all $x\in X$. If $A$ is a purely infinite \CK{} over $X$, then, for all $x\in X$, the groups $\FK_{\si x}^1(A)$ and $\FK_{\osi x}^0(A)$ are furthermore finitely generated, and the rank of $\FK_{\si x}^1(A)$ coincides with the rank of the cokernel of the map $i_{\obd x}^{\osi x}\colon \FK_{\obd x}^0(A)\to \FK_{\osi x}^0(A)$.
\end{proposition}
\begin{proof}

Exactness of $\FKgunnar(A)$ is stated in Corollary~\ref{cor:FKgunnarexact}.
The group $\FK_{\si x}^1(A)$ is free for all $x\in X$ since the $\K_1$-group of a graph \csa{} is free, and a subquotient of a real-rank-zero graph \csa{} is Morita equivalent to a graph \csa{} \cite{raeburn}*{Theorem~4.9}.

Assume that $A$ is a \CK{}.
For any \CK{} $B$, $\K_*(B)$ is finitely generated and $\rank\K_0(B)=\rank\K_1(B)$. Since a subquotient of a purely infinite \CK{} is stably isomorphic to a \CK{}, the groups $\FK_{\si x}^1(A)$ and $\FK_{\osi x}^0(A)$ are finitely generated and $\rank\FK_{\si x}^1(A)=\rank\FK_{\si x}^0(A)$ for all $x\in X$.
Since $A$ has real rank zero, the sequence
\[ \FK_{\obd x}^0(A) \xrightarrow{i_{\obd x}^{\osi x}} \FK_{\osi x}^0(A) \xrightarrow{r_{\osi x}^{\si x}} \FK_{\si x}^0(A) \rightarrow 0\]
is exact by~\cite{brownpedersen}*{Theorem~3.14}. Hence
\[
\rank\FK_{\si x}^1(A)=\rank\Bigl(\coker\bigl(i_{\obd x}^{\osi x}\colon\FK_{\obd x}^0(A)\to\FK_{\osi x}^0(A)\bigr)\Bigr).\qedhere
\]
\end{proof}

Combining Proposition~\ref{prop:range} with Theorem~\ref{thm:range}, one obtains a complete description of the range of reduced filtered $\K$\nb-theory  on purely infinite tight graph \csas{} over~$X$, and on purely infinite tight \CK{}s over $X$.

\begin{theorem} \label{thm:range}
Let $M$ be an exact $\Catgunnar$\nb-mod\-ule  with $M(\gsi x)$ free for all $x\in X$.
Then there exists a countable graph $E$ such that all vertices in $E$ are regular and support at least two loops, the \csa{} $C^*(E)$ is tight over $X$ and $\FKgunnar\bigl(C^*(E)\bigr)$ is isomorphic to~$M$.
By construction $C^*(E)$ is purely infinite.

The graph $E$ can be chosen to be finite if and only if $M(\gsi x)$ and $M(\gosi x)$ are finitely generated and the rank of $M(\gsi x)$ coincides with the rank of the cokernel of $i\colon M(\gobd x)\to M(\gosi x)$ for all $x\in X$.  
If $E$ is chosen finite, then by construction $C^*(E)$ is a Cuntz--Krie\-ger algebra.

\begin{proof}
For each $x\in X$, we may choose, by~\cite{ektw}*{Proposition 3.3}, a countable, non-empty set~$V_x$, a matrix $D_x\in M_{V_x}(\Z_+)$ and isomorphisms $\phi_{\gsi x}\colon M(\gsi x)\to \ker D_x$ and $\phi_{\rsi x}\colon M(\rsi x)\to \cok D_x$, where $M(\rsi x) \defeq \cok(M(\gobd x)\xrightarrow{i}M\bigl(\gosi x)\bigr)$. 
Define $r_{\gosi x}^{\rsi x}\colon M(\gosi x)\to M(\rsi x)$ as the cokernel map.
Given a matrix $D$, we let $E(D)$ denote the graph with adjacency matrix $D^t$.
We may furthermore assume that all vertices in the graph $E({1+D_x})$ are regular and support at least two loops. If $M(\gsi x)$ and $M(\gosi x)$ are finitely generated and the rank of $M(\gsi x)$ coincides with the rank of the cokernel of $i\colon M(\gobd x)\to M(\gosi x)$, then the set $V_x$ can be chosen to be finite.

For each $y,z\in X$ with $y\neq z$ we desire to construct a matrix $H_{yz}\colon\Z^{V_z}\to\Z^{V_y}$ with non-negative entries satisfying that $H_{yz}$ is non-zero if and only if $y>z$, and satisfying that for each $x\in X$ there exist isomophisms $\phi_{\gobd x}$ and $\phi_{\gosi x}$ making the diagrams
\begin{equation}
\begin{split}
\xymatrix{
M(\gobd x)\ar[rd]^-{\phi_{{\gobd x}}}\ar[rr]^-{i_{\gobd x}^{\gosi x}} & & M(\gosi x)\ar[d]^-{\phi_{\gosi x}}\ar[rr]^-{r_{\gosi x}^{\rsi x}} & & M(\rsi x)\ar[dl]^-{\phi_{\rsi x}} \\
& \cok D_{\ob \{x\}} \ar[r]^-{\iota_{\obd x}^{\osi x}} & \cok D_{\osi x} \ar[r]^-{\pi_{\osi x}^{\si x}} & \cok D_{x} \ar[d]^-{0} & \\
& \ker D_{x} \ar[u]^-{D_{\osi x}|_{\phi^{-1}(x)}^{\phi^{-1}(\ob \{x\})}} & \ker D_{\osi x} \ar[l]^-{\pi_{\osi x}^{\si x}} & \ker D_{\ob \{x\}} \ar[l]^-{\iota_{\obd x}^{\osi x}} & \\
M(\gsi x)\ar[ru]^-{\phi_{\gsi x}}\ar[uuu]^-{\delta_{\gsi x}^{\gobd x}} &&&&
} \label{FKx}
\end{split}
\end{equation}
and
\begin{equation}
\begin{split}
\xymatrix@C+20pt{
M(\gosi y)\ar[dr]^-{\phi_{\gosi y}}\ar[rr]^-{i_{\gosi y}^{\gobd x}} && M(\gobd x)\ar[d]^-{\phi_{\gobd x}} & \\
& \cok D_{\osi y} \ar[r]^-{\iota_{\osi y}^{\obd x}} & \cok D_{\obd x} \ar[r]^-{\pi_{\obd x}^{\obd x \setminus \osi y}} & \cok D_{\ob \{x\}\setminus \osi y} \ar[d]^-{0} \\
& \ker D_{\ob \{x\}\setminus \osi y} \ar[u]^-{D_{\ob \{x\}}|_{\phi^{-1}(\ob \{x\}\setminus\osi y)}^{\phi^{-1}(\osi y)}} & \ker D_{\obd x} \ar[l]^-{\pi_{\obd x}^{\obd x \setminus \osi y}} & \ker D_{\osi y} \ar[l]^-{\iota_{\osi y}^{\obd x}}
} \label{FKyx}
\end{split}
\end{equation}
commute when $y\to x$, where $D_F\in M_{V_F}(\Z_+)$ for each $F\subseteq X$ is defined as
\[ D_F(v,w) = \begin{cases} D_x(v,w) & v,w\in V_x \\ H_{yz}(v,w) & v\in V_y, w\in V_x, x\neq y \end{cases} \]
with $V_F=\bigcup_{y\in F}V_y$.
The constructed graph $E({D_X+1})$ then has the desired properties.

We proceed by a recursive argument, by adding to an open subset an open point in the complement. Given $U\in\Op(X)$, assume that for all $z,y\in U$, the matrices $H_{yz}$ and isomorphisms $\phi_{\gobd y}$ and $\phi_{\gosi y}$ have been defined and satisfy that the diagrams~\eqref{FKx} and~\eqref{FKyx} commute for all $x,y\in U$ with $y\to x$.
Let $x$ be an open point in $X\setminus U$ and let us construct isomorphisms $\phi_{\gobd x}$ and $\phi_{\gosi x}$, and for all $y\in\ob \{x\}$ non-zero matrices $H_{yx}$, making the diagrams~\eqref{FKx} and~\eqref{FKyx} commute.

Consider the commuting diagram
\[ \xymatrix@C+10pt{
\displaystyle\bigoplus_{(p,q)\in\DoublePaths(x)} M(\gosi{s(p,q)}) \ar[r]^-{\scriptsize\begin{pmatrix} i_p & -i_q \end{pmatrix}}\ar[d]^-{(\phi_{\gosi{s(p,q)}})} & \displaystyle\bigoplus_{y\to x} M(\gosi y) \ar[r]^-{(i_{\gosi y}^{\gobd x})} \ar[d]^-{(\phi_{\gosi y})} & M(\gobd x) \ar[r]\ar@{..>}[d]^-{(\phi_{\gobd x})} & 0 \\
\displaystyle\bigoplus_{(p,q)\in\DoublePaths(x)} \cok D_{\osi{s(p,q)}} \ar[r]^-{\scriptsize\begin{pmatrix} \iota_p & -\iota_q \end{pmatrix}} & \displaystyle\bigoplus_{y\to x} \cok D_{\osi y} \ar[r]^-{(\iota_{\osi y}^{\obd x})} & \cok D_{\ob \{x\}} \ar[r] & 0 .
} \]
The top row is exact by exactness of $M$, and the bottom row is exact by exactness of $\FK\bigl(C^*(E({1+D_{\ob \{x\}}}))\bigr)$.  An isomorphism $\phi_{\gobd x}\colon M(\gobd x)\to \cok {D_{\ob \{x\}}}$ is therefore induced.  By construction,~\eqref{FKyx} commutes for all $y\to x$.

Now consider the commuting diagram
\[ \xymatrix{
M(\gobd x) \ar[rr]^-{i_{\gobd x}^{\gosi x}} \ar[dr]^-{\phi_{\gobd x}} && M(\gosi x) \ar[rr]^-{r_{\gosi x}^{\rsi x}} && M(\rsi x)\ar[dl]^-{\phi_{\rsi x}}  \\
&\cok D_{\ob \{x\}}\ar[r] & M(\gosi x) \ar@{=}[u]\ar[r] & \cok D_x \ar[d]^-{0} & \\
& \ker D_x\ar[u]  & F \ar[l] & \ker D_{\ob \{x\}} \ar[l] & \\
M(\gsi x) \ar[uuu]^-{\delta_{\gsi x}^{\gobd x}} \ar[ur]^-{\phi_{\gsi x}} && &&
} \]
with the maps in the inner sequence being the unique maps making the squares commute, and 
where a free group~$F$ and maps into and out of it have been chosen so that the inner six-term sequence is exact.
Apply Theorem~\ref{ektw} to the inner six-term exact sequence to get non-zero matrices $H_{yx}$ for all $y\in\ob \{x\}$ realizing the sequence, that is, making~\eqref{FKx} commute.

Finally, we note that the constructed graph algebra $C^*(E(D_X+1))$ is purely infinite by~\cite{hongszymanski}*{Theorem~2.3} since all vertices in $E(D_X+1)$ are regular and support two loops.
Since the graph $E(D_X+1)$ has no sinks or sources, the graph algebra $C^*(E(D_X+1))$ is a \CK{} when $E(D_X+1)$ is finite.
\end{proof}
\end{theorem}

Combining the previous theorem with Restorff's classification of purely infinite \CKs{}~\cite{restorff}, we obtain the following description of stable isomorphism classes of purely infinite \CKs{}.

\begin{corollary}
The functor~$\FKgunnar$ induces a bijection between the set of stable isomorphism classes of tight purely infinite \CKs{} over~$X$ and the set of isomorphism classes of exact $\Catgunnar$\nb-mod\-ules $M$ such that, for all $x\in X$, $M(\gsi x)$ is free, $M(\gsi x)$ and $M(\gosi x)$ are finitely generated, and the rank of $M(\gsi x)$ coincides with the rank of the cokernel of the map $i_{\gobd x}^{\gosi x}\colon M(\gobd x)\to M(\gosi x)$.
\end{corollary}

\section{Unital reduced filtered \texorpdfstring{$\K$}{K}-theory} \label{sec:unital}

Anticipating a generalization of the main result in \cite{restorff} accounting for actual isomorphisms rather than stable isomorphisms, we also provide a `unital' version of our range result.
Depending on the space $X$, the group $\K_0(A)$ may not be part of the invariant $\FKgunnar(A)$. This slightly complicates the definition of unital reduced filtered $\K$-theory.

For $x,x'\in X$, we let $\inf(x,x')$ denote the set $\{ y\in X \mid y\to x, y\to x'\}$.

\begin{definition}
The category $\Mod(\Catgunnar)^\pt$ of \emph{pointed $\Catgunnar$\nb-mod\-ules} is defined to have objects given by pairs $(M,m)$ where $M$ is an $\Catgunnar$\nb-mod\-ule and
\[
m\in\coker\left(\bigoplus_{\stackrel{x,x'\in X}{y\in\inf(x,x')}}M(\gosi y)\xrightarrow{\begin{pmatrix} i_{\gosi y}^{\gobd x}i_{\gobd x}^{\gosi x} & -i_{\gosi y}^{\gobd x'}i_{\gobd x'}^{\gosi x'} \end{pmatrix}} \bigoplus_{x\in X} M(\gosi x)\right),
\]
and a morphism $\phi\colon(M,m)\to(N,n)$ is an $\Catgunnar$\nb-mod\-ule homomorphism from~$M$ to~$N$ whose induced map on the cokernels sends~$m$ to~$n$.
\end{definition}

\begin{lemma} \label{lem:unitexact}
Let $A$ be a unital \csa{} over $X$ of real rank zero, and let $U\in\Op(X)$.  Then the sequence
\[
\bigoplus_{\stackrel{x,x'\in U}{y\in\inf(x,x')}} \FK_{\osi y}^0(A)\xrightarrow{\begin{pmatrix} i_{\osi y}^{\osi x} & -i_{\osi y}^{\osi {x'}} \end{pmatrix}}\bigoplus_{x\in U} \FK_{\osi x}^0(A) \xrightarrow{(i_{\osi x}^U)} \FK_U^0(A) \rightarrow 0
\]
is exact.
\end{lemma}
\begin{proof}
This follows from a twofold application of Lemma~\ref{lem:exactseq} using that $U$ is covered by $(\osi x)_{x\in U}$ and that $\osi x\cap \osi{x'}$ is covered by $(\osi y)_{y\in\inf(x,x')}$.
\end{proof}

\begin{definition}
Let $A$ be a unital \csa{} over~$X$ with real rank zero. The \emph{unital reduced filtered $\K$\nb-theory} $\FKgunnar^\unit(A)$ is defined as the pointed $\Catgunnar$\nb-mod\-ule $(\FKgunnar(A),u(A))$ where $u(A)$ is the unique element in
\[ \coker\left( \bigoplus_{\stackrel{x,x'\in X}{y\in\inf(x,x')}} \FK_{\osi y}^0(A)\xrightarrow{\begin{pmatrix} i_{\osi y}^{\osi x} & -i_{\osi y}^{\osi {x'}} \end{pmatrix}}\bigoplus_{x\in X} \FK_{\osi x}^0(A)\right)\]
that is mapped to $[1_A]$ in $\K_0(A)$ by the map induced by the family $\bigl(\FK_{\osi x}^0(A)\xrightarrow{(i_{\osi x}^X)}\FK_X^0(A)\bigr)_{x\in X}$, see~Lemma~\ref{lem:unitexact}.
\end{definition}

\begin{lemma} \label{lem:extend}
Let $A$ and $B$ be \csa s over $X$ of real rank zero, and let $U\in\Op(X)$.  Let a family of isomorphisms $\phi_{\osi x}^0\colon\FK_{\osi x}^0(A)\to\FK_{\osi x}^0(B)$, ${x\in U}$ be given and assume that $\phi_{\osi y}^0i_{\osi y}^{\osi x}=i_{\osi y}^{\osi x}\phi_{\osi x}^0$ holds for all $x,y\in U$ with $y\to x$.
Then $(\phi_{\osi x}^0)_{x\in U}$ can be uniquely extended to a family  of isomorphisms $\phi^0_Y\colon\FK_Y^0(A)\to\FK_Y^0(B)$, ${Y\in\LC(U)}$, that commute with the natural transformations $i$ and $r$.
\end{lemma}
\begin{proof}
We may assume that $U=X$. The part of the construction in the proof \cite{reduction}*{Theorem~5.17} that involves only groups in even degree makes sense when $X$ is an arbitrary finite $T_0$\nb-space and implies the present claim as a corollary.
\end{proof}

For a unital graph \csa{} $C^*(E)$, the class of the unit $[1_{C^*(E)}]$ in $\K_0\bigl(C^*(E)\bigr)$ is sent, via the canonical isomorphism $\K_0\bigl(C^*(E)\bigr)\to\coker D_E$, to the class $[\mathbb 1]=\begin{pmatrix} 1 & 1 & \cdots & 1 \end{pmatrix} + \im D_E$, where $1+D_E^t$ denotes the adjacency matrix for~$E$,~\cite{tomforde}*{Theorem 2.2}.
Using this and the formula of~\cite{cet}, see~\S\ref{ssec:calc}, the unital reduced filtered $\K$\nb-theory of a unital graph \csa{} can be calculated.
A graph \csa{} $C^*(E)$ is unital if and only if its underlying graph $E$ has finitely many vertices.  So by the formula of \cite{cet}, a unital graph \csa{} and its subquotients always have finitely generated $\K$\nb-theory.

\begin{theorem} \label{thm:range_unit}
Let $X$ be a finite $T_0$-space, and let $(M,m)$ be an exact pointed $\Catgunnar$\nb-mod\-ule.  Assume that for all $x\in X$, $M(\gsi x)$ is a free abelian group,
\[ \coker (M(\gobd x)\xrightarrow{i_{\gobd x}^{\gosi x}} M\bigl(\gosi x)\bigr) \]
 is finitely generated, and $\rank M(\gsi x)\leq \rank \coker (M(\gobd x)\xrightarrow{i_{\gobd x}^{\gosi x}} M\bigl(\gosi x)\bigr)$.

Then there exists a countable graph $E$ such that all vertices in $E$ support at least two loops, the set $E^0$ is finite, the \csa{} $C^*(E)$ is tight over $X$ and the pointed $\Catgunnar$\nb-mod\-ule $\FKgunnar^\unit\bigl(C^*(E)\bigr)$ is isomorphic to $(M,m)$. By construction $C^*(E)$ is unital and purely infinite.

The graph $E$ can be chosen such that all of its vertices are regular if and only if the rank of $M(\gsi x)$ coincides with the rank of the cokernel of $i_{\gobd x}^{\gosi x}\colon M(\gobd x)\to M(\gosi x)$ for all $x\in X$.  
If $E$ is chosen to have regular vertices, then by construction $C^*(E)$ is a Cuntz--Krie\-ger algebra.
\end{theorem}
\begin{proof}
The proof is carried out by the same strategy as the proof of Theorem~\ref{thm:range}.  However, we here construct graphs that may have singular vertices.  We refer to~\cite{cet} for the formula for calculating six-term exact sequences in $\K$\nb-theory for such graph \csa{}s.

By Theorem~\ref{thm:range} there exists a \csa{} $B$ of real rank zero with $M\cong\FKgunnar(B)$, so we may assume that $M=\FKgunnar(B)$.
Since, by Lemma~\ref{lem:unitexact},
\[
\coker\Bigl(\bigoplus_{y\in\inf(x,x')}M(\gosi y)\xrightarrow{\begin{pmatrix} i_{\gosi y}^{\gosi x} & -i_{\gosi y}^{\gosi x'} \end{pmatrix}} \bigoplus_{x\in X} M(\gosi x,0)\Bigr)
\]
is isomorphic to $\FK_X^0(B)=\K_0(B)$, we may identify $m$ with its image in $\FK_X^0(B)$ but note that this image may not be~$[1_B]$.

The construction is similar to the construction in the proof of Theorem~\ref{thm:range}. Let $x_1$ be an open point in $X$ and define $U_1=\{x_1\}$.  Define $U_k$ recursively by choosing an open point $x_k$ in $X\setminus U_{k-1}$ and defining $U_k=U_{k-1}\cup\{x_k\}$.  Let $C_k$ denote the largest subset of $U_k$ that is closed in $X$.  Observe that
\[ C_k = X\setminus \bigcup_{y\in\clo(X)\setminus U_k} \osi y \]
where $\clo(X)$ denotes the closed points in $X$.  And observe that if $x_k$ is closed in $X$ then
$C_k\setminus C_{k-1}\subseteq \osi{x_k}$ and $C_k\setminus\osi{x_k}= C_{k-1}$, and otherwise $C_k=C_{k-1}$.
Define for all closed subsets $C$ of $X$ the element $m_C$ of $\FK_C^0(B)$ as the image of $m$ under $r\colon \FK_X^0(B)\to \FK_C^0(B)$.

For each $x$ not closed in $X$, choose by Proposition~3.6 of~\cite{ektw} a graph $E_x$ that is transitive, has finitely many vertices that all support at least two loops, and such that $\K_1\bigl(C^*(E_x)\bigr)$ is isomorphic to $\FK_{\si x}^1(B)$ and $\K_0\bigl(C^*(E_x)\bigr)$ is isomorphic to $\FK_{\si x}^0(B)$.
Define $V_{\si x}=E_x^0$ and $V'_{\si x}=(E_x^0)_{\textnormal{reg}}$,
let $D_{\si x}\in M_{V_{\si x}}(\Z_+\cup\{\infty\})$ such that $1+D^t_{\si x}$ is the adjacency matrix for $E_x$, and let $D'_{\si x}$ denote the $V'_{\si x}\times V_{\si x}$ matrix defined by $D'_{\si x}(v,w)=D_{\si x}(v,w)$.
If $\rank \FK_{\si x}^1(B)=\rank \FK_{\si x}^0(B)$ then $V_{\si x}=V'_{\si x}$.
Let isomorphisms $\phi_{\si x}^1\colon \FK_{\si x}^1(B)\to\ker D'_{\si x}$ and $\phi_{\si x}^0\colon \FK_{\si x}^0(B)\to\coker D'_{\si x}$ be given.
For $x$ closed in $X$ we may by Proposition~3.6 of~\cite{ektw} choose $E_x$ and $\phi_{\si x}^0$ such that furthermore $\phi_{\si x}^0(m_{\si x})=[\mathbb 1]$.

Define
\[ V_U=\bigcup_{x\in U}V_{\si x},  V'_U=\bigcup_{x\in U}V'_{\si x} \]
for all $U\in\Op(X)$.
As in the proof of Theorem~\ref{thm:range} we wish to construct for all $x,y\in X$ with $x\neq y$, $V_{\si y}'\times V_{\si x}$ matrices $H_{yx}'$ over $\Z_+$ with $H_{yx}'\neq 0$ if and only if $y>x$.
When having constructed such $H_{yx}'$, we construct a $V_{\si y}\times V_{\si x}$ matrix $H_{yx}$ over $\Z_+\cup\{\infty\}$ by
\[ H_{yx}(v,w) = \begin{cases}
H_{yx}'(v,w) & v\in V_{\si y}' \\
\infty & v\in V_y\setminus V_{\si y}' \textnormal{ and }  y>x \\
0 & v\in V_y\setminus V_{\si y}' \textnormal{ and } y<x .
\end{cases} \]
We then define for $U\in\Op(X)$, a $V_U'\times V_U$ matrix $D_U'$ over $\Z_+$ and a matrix $D_U\in M_{V_U}(\Z_+\cup\{\infty\})$ by
\[ D_U'(v,w) = \begin{cases}
D_{\si x}'(v,w) & v\in V_{\si x}', w\in V_{\si x} \\
H_{yx}'(v,w) & v\in V_{\si y}', w\in V_{\si x}
\end{cases} \]
and 
\[ D_U(v,w) = \begin{cases}
D_{\si x}(v,w) & v,w\in V_{\si x} \\
H_{yx}(v,w) & v\in V_{\si y}, w\in V_{\si x} .
\end{cases} \]

For a matrix $D$, we denote by $E(D)$ the graph with adjacency matrix $D^t$.
Since the graph $E({1+D_{\si x}})$ has a simple graph \csa{}, it cannot contain breaking vertices.  By construction, neither will $E(1+{D_X})$.  So $E({1+D_X})$ will be tight over $X$ as in the proof of Theorem~\ref{thm:range}.  The matrices $(H_{yx_k}')_{y\in\obd{x_k}}$ will be constructed recursively over $k\in\{1,\ldots, n\}$ so that the following holds:
For each $x\in X$, there are isomorphisms making the diagram
\begin{equation}
\begin{split}
\xymatrix{
\FK_{\si x}^1(B)\ar[r]^-{\delta_{\si x}^{\obd x}}\ar[d]^-{\phi_{\si x}^1}_-{\cong} & \FK_{\obd x}^0(B)\ar[r]^-{i_{\obd x}^{\osi x}}\ar[d]^-{\phi_{\obd x}^0}_-{\cong} & \FK_{\osi x}^0(B)\ar[d]^-{\phi_{\osi x}^0}_-{\cong} \\
\ker D'_{\si x}\ar[r]^-{\Delta_{\si x}^{\obd x}} & \coker D'_{\obd x}\ar[r]^-{\iota_{\obd x}^{\osi x}} & \coker D'_{\osi x}
}
\end{split}
\label{unitrange:FKx}
\end{equation}
commute and satisfying for all $y\to x$ that
\begin{equation}
\begin{split}
\xymatrix{
\FK_{\osi y}^0(B)\ar[r]^-{i_{\osi y}^{\obd x}}\ar[d]^-{\phi_{\osi y}^0}_-{\cong} & \FK_{\obd x}^0(B)\ar[d]^-{\phi_{\obd x}^0}_-{\cong} \\
\coker D'_{\osi y}\ar[r]^-{\iota_{\osi y}^{\obd x}} & \coker D'_{\obd x}
}
\end{split}
\label{unitrange:FKxy}
\end{equation}
commutes, and for all $k\in\{1,\ldots,n\}$ that the isomorphism
\[ \phi_{C_k}^0\colon \FK_{C_K}^0(B)\to\coker D'_{C_k} \]
induced by $(\phi_{\osi x}^0)_{x\in X}$, see Lemma~\ref{lem:extend}, sends $m_{C_k}$ to $[\mathbb 1]$.

Assume that the matrices $(H_{yx_i}')_{y\in\obd{x_i}}$ have been constructed for all $i<k$.  Then isomorphisms $(\phi_Y^0)_{Y\in\LC(U_{k-1})}$ are induced by Lemma~\ref{lem:extend}.

Assume that $x_k$ is a closed point in $X$.
Since $D_{C_k\setminus\si{x_k}}$ is already defined, we may define $1_{C_k\setminus\osi{x_k}}$ as the element in $\Z^{V_{C_k\setminus\si{x_k}}}$ with
\[ 1_{C_k\setminus\osi{x_k}}(i) =
\begin{cases}
1 & \textnormal{if }i\in V_{C_k\setminus\osi{x_k}} \\
0 & \textnormal{if }i\in V_{(C_k\cap\osi{x_k})\setminus\si{x_k}}.
\end{cases} \]
Define $\widetilde m_{C_k\setminus\si{x_k}}$ as the preimage of $1_{C_k\setminus\osi{x_k}}$ under the isomorphism $\phi_{C_k\setminus\si{x_k}}^0$.
Notice that $m_{C_k}r_{C_k}^{C_{k-1}}\phi_{C_{k-1}}^0=[\mathbb 1]$ and that, since $C_{k-1}$ is closed in $C_k\setminus\si{x_k}$,
\begin{align*}
\widetilde m_{C_k\setminus\si{x_k}}i_{C_k\setminus\si{x_k}}^{C_k}r_{C_k}^{C_{k-1}}\phi_{C_{k-1}}^0
&=\widetilde m_{C_k\setminus\si{x_k}}r_{C_k\setminus\si{x_k}}^{C_{k-1}}\phi_{C_{k-1}}^0 \\
&= \widetilde m_{C_k\setminus\si{x_k}}\phi_{C_k\setminus\si{x_k}}^0r_{C_k\setminus\si{x_k}}^{C_{k-1}} \\
&= 1_{C_k\setminus\osi{x_k}}r_{C_k\setminus\si{x_k}}^{C_{k-1}} = [\mathbb 1] ,
\end{align*}
so by injectivity of the map $\phi_{C_{k-1}}^0$, the element $m_{C_k}-\widetilde m_{C_k\setminus\si{x_k}}i_{C_k\setminus\si{x_k}}^{C_k}$ lies in $\ker r_{C_k}^{C_{k-1}}=\im i_{C_k\cap\osi{x_k}}^{C_k}$.
Choose $\widetilde m_{C_k\cap\osi{x_k}}$ in $\FK_{C_k\cap\osi{x_k}}^0(B)$ such that
\[ m_{C_k} =  m_{C_k\cap\osi{x_k}}i_{C_k\cap\osi{x_k}}^{C_k}  + \widetilde m_{C_k\setminus\si{x_k}}i_{C_k\setminus\si{x_k}}^{C_k} . \]
Choose $\widetilde m_{\osi{x_k}}$ in $\FK_{\osi{x_k}}^0(B)$ such that
\[
\widetilde m_{\osi{x_k}}r_{\osi{x_k}}^{C_k\cap\osi{x_k}} = \widetilde m_{C_k\cap\osi{x_k}} .
\]
Consider the diagram
\begin{equation}
\begin{split}
\xymatrix{
\FK_{\si{x_k}}^1(B) \ar[r]^-{\delta_{\si{x_k}}^{\obd{x_k}}}\ar[d]^-{\phi_{\si{x_k}}^1}_-{\cong} & \FK_{\obd{x_k}}^0(B)\ar[r]^-{i_{\obd{x_k}}^{\osi{x_k}}}\ar[d]^-{\phi_{\obd{x_k}}^0}_-{\cong} & \FK_{\osi{x_k}}^0(B)\ar[r]^-{r_{\osi{x_k}}^{\si{x_k}}}\ar@{.>}[d]^-{\phi_{\osi{x_k}}^0} & \FK_{\si{x_k}}^0(B) \ar[d]^-{\phi_{\si{x_k}}^0}_-{\cong} \\
\ker D'_{x_k} \ar[r] & \coker D'_{\obd{x_k}}\ar@{.>}[r] & \coker D'_{\osi x}\ar@{.>}[r] & \coker D'_{\si{x_k}},
} \label{unitrange:sixterm}
\end{split}
\end{equation}
where $\widetilde m_{\osi{x_k}}$ is mapped to $m_{\si{x_k}}$ which by $\phi_{\si{x_k}}^0$ is mapped to $[\mathbb 1]$.
As in the proof of Theorem~\ref{thm:range} we apply Theorem~\ref{ektw} to construct $D'_{\osi{x_k}}$ from $D'_{\si{x_k}}$ and $D'_{\obd{x_k}}$ by constructing nonzero matrices $H_{yx_k}$ for all $y\in\obd{x_k}$.  By Proposition~4.8 of~\cite{ektw} we may furthermore achieve that
$m_{\osi{x_k}}\phi_{\osi{x_k}}=[\mathbb 1]$.

That \eqref{unitrange:FKx} and \eqref{unitrange:FKxy} hold for $x_k$ follows immediately from the construction.  To verify that the map $\phi_{C_k}^0$ induced by $(\phi_{\osi y}^0)_{y\in U_k}$ satisfies $m_{C_k}\phi_{C_k}^0=[\mathbb 1]$, observe that the map $\phi_{C_k\cap\osi{x_k}}^0$ induced by $\phi_{\osi{x_k}}^0$ will map $\widetilde m_{C_k\cap\osi{x_k}}$ to $[\mathbb 1]$, and consider the commuting diagram
\[ \xymatrix@C+50pt{
\FK_{C_k\setminus\si{x_k}}^0(B)\oplus \FK_{C_k\cap\osi{x_k}}^0(B) \ar[r]^-{\scriptsize\begin{pmatrix} i_{C_k\setminus\si{x_k}}^{C_k} \\ i_{C_k\cap\osi{x_k}}^{C_k} \end{pmatrix}}\ar[d]^-{\phi_{C_k\setminus\si{x_k}}^0\oplus\phi_{C_k\cap\osi{x_k}}^0}_-{\cong} & \FK_{C_k}^0(B)\ar[d]^-{\phi_{C_k}^0}_-{\cong} \\
\coker D'_{C_k\setminus\si{x_k}}\oplus \coker D'_{C_k\cap\osi{x_k}} \ar[r]^-{\scriptsize\begin{pmatrix}  \iota_{C_k\setminus\si{x_k}}^{C_k} \\ \iota_{C_k\cap\osi{x_k}}^{C_k}  \end{pmatrix}} & \coker D'_{C_k}  .
}  \]
Since $(\widetilde m_{C_k\setminus\si{x_k}}, \widetilde m_{C_k\cap\osi{x_k}})$ is mapped to $m_{C_k}$ by $\begin{pmatrix}  i_{C_k\setminus\si{x_k}}^{C_k} \\ i_{C_k\cap\osi{x_k}}^{C_k}  \end{pmatrix}$ and to $(1_{C_k\setminus\osi{x_k}},[\mathbb 1])$ by $\phi_{C_k\setminus\si{x_k}}^0\oplus\phi_{C_k\cap\osi{x_k}}^0$, we see by commutativity of the diagram that $m_{C_k}\phi_{C_k}^0=[\mathbb 1]$.

For $k$ with $x_k$ not closed in $X$, $C_k$ equals $C_{k-1}$ and a construction similar to the one in the proof of Theorem~\ref{thm:range} applies. 
As in the proof of Theorem~\ref{thm:range}, Proposition~4.7 of~\cite{ektw} allows us to make sure that $H'_{yx_k}\neq 0$ when $y\in\obd{x_k}$.

Finally, we note that the constructed graph algebra $C^*(E(D_X+1))$ is purely infinite by~\cite{hongszymanski}*{Theorem~2.3} since all vertices in $E(D_X+1)$ support two loops and $E(D_X+1)$ has no breaking vertices.
Since the graph $E(D_X+1)$ has no sinks or sources, the graph algebra $C^*(E(D_X+1))$ is a \CK{} when $E(D_X+1)$ is finite.
\end{proof}


\begin{bibsection}
    \begin{biblist}

\bib{reduction}{article}{
  author = {Arklint, Sara E.},
  author = {Bentmann, Rasmus},
  author = {Katsura, Takeshi},
  title = {Reduction of filtered {$K$}-theory and a characterization of {C}untz--{K}rieger algebras},
  year = {2013},
  eprint = {arXiv:1301.7223v3},
  }

\bib{arklint_ruiz}{article}{
   author = {{Arklint}, Sara E.},
   author={{Ruiz}, Efren},
    title = {Corners of Cuntz--Krie\-ger algebras},
   eprint = {arXiv:1209.4336},
     year = {2012},
}

\bib{boyle_huang}{article}{
    AUTHOR = {Boyle, Mike},
    author ={Huang, Danrun},
     TITLE = {Poset block equivalence of integral matrices},
   JOURNAL = {Trans. Amer. Math. Soc.},
    VOLUME = {355},
      YEAR = {2003},
    NUMBER = {10},
     PAGES = {3861--3886 (electronic)},
      ISSN = {0002-9947},
       DOI = {10.1090/S0002-9947-03-02947-7},
       URL = {http://dx.doi.org/10.1090/S0002-9947-03-02947-7},
}


\bib{Bredon}{article}{
   author={Bredon, Glen E.},
   title={Cosheaves and homology},
   journal={Pacific J. Math.},
   volume={25},
   date={1968},
   pages={1--32},
   issn={0030-8730},
}

\bib{brownpedersen}{article}{
  author={Brown, Lawrence G.},
      AUTHOR = {Pedersen, Gert K.},
     TITLE = {{$C^*$}-algebras of real rank zero},
   JOURNAL = {J. Funct. Anal.},
    VOLUME = {99},
      YEAR = {1991},
    NUMBER = {1},
     PAGES = {131--149},
      ISSN = {0022-1236},
       DOI = {10.1016/0022-1236(91)90056-B},
       URL = {http://dx.doi.org/10.1016/0022-1236(91)90056-B},
}

\bib{cet}{article}{
    AUTHOR = {Carlsen, Toke Meier},
    author= {Eilers, S{\o{}}ren},
    author={Tomforde, Mark},
     TITLE = {Index maps in the {$\K$}-theory of graph \csas{}},
JOURNAL = {J. K-Theory},
    VOLUME = {9},
      YEAR = {2012},
    NUMBER = {2},
     PAGES = {385--406},
      ISSN = {1865-2433},
       DOI = {10.1017/is011004017jkt156},
       URL = {http://dx.doi.org/10.1017/is011004017jkt156},
}

\bib{cuntz}{article}{
   author={Cuntz, Joachim},
   title={A class of $C^{\ast} $-algebras and topological Markov chains.
   II. Reducible chains and the Ext-functor for $C^{\ast} $-algebras},
   journal={Invent. Math.},
   volume={63},
   date={1981},
   number={1},
   pages={25--40},
   issn={0020-9910},
   doi={10.1007/BF01389192},
}

\bib{cuntz_krieger}{article}{
    AUTHOR = {Cuntz, Joachim},
    author ={Krieger, Wolfgang},
     TITLE = {A class of {$C^{\ast} $}-algebras and topological {M}arkov
              chains},
   JOURNAL = {Invent. Math.},
    VOLUME = {56},
      YEAR = {1980},
    NUMBER = {3},
     PAGES = {251--268},
      ISSN = {0020-9910},
       DOI = {10.1007/BF01390048},
       URL = {http://dx.doi.org/10.1007/BF01390048},
}

\bib{drinentomforde}{article}{
    AUTHOR = {Drinen, Douglas},
    author = {Tomforde, Mark},
     TITLE = {The {$C^*$}-algebras of arbitrary graphs},
   JOURNAL = {Rocky Mountain J. Math.},
    VOLUME = {35},
      YEAR = {2005},
    NUMBER = {1},
     PAGES = {105--135},
      ISSN = {0035-7596},
       DOI = {10.1216/rmjm/1181069770},
       URL = {http://dx.doi.org/10.1216/rmjm/1181069770},
}

\bib{ektw}{article}{
author = {Eilers, S{\o{}}ren},
author={Katsura, Takeshi},
author={Tomforde, Mark},
author={West, James},
title={The ranges of {$\K$}-theoretic invariants for non-simple graph \csas{}},
eprint = {arXiv:1202.1989v1},
year={2012},
}

\bib{hongszymanski}{article}{
    AUTHOR = {Hong, Jeong Hee},
    AUTHOR ={Szyma{\'n}ski, Wojciech},
     TITLE = {Purely infinite {C}untz-{K}rieger algebras of directed graphs},
   JOURNAL = {Bull. London Math. Soc.},
    VOLUME = {35},
      YEAR = {2003},
    NUMBER = {5},
     PAGES = {689--696},
      ISSN = {0024-6093},
       DOI = {10.1112/S0024609303002364},
       URL = {http://dx.doi.org/10.1112/S0024609303002364},
}

\bib{MN:Filtrated}{article}{
  author={Meyer, Ralf},
  author={Nest, Ryszard},
  title={$C^*$-algebras over topological spaces: filtrated $\K$\nb-theory},
JOURNAL = {Canad. J. Math.},
    VOLUME = {64},
      YEAR = {2012},
    NUMBER = {2},
     PAGES = {368--408},
      ISSN = {0008-414X},
       DOI = {10.4153/CJM-2011-061-x},
       URL = {http://dx.doi.org/10.4153/CJM-2011-061-x},
}

\bib{MN:Bootstrap}{article}{
  author={Meyer, Ralf},
  author={Nest, Ryszard},
  title={$C^*$-algebras over topological spaces: the bootstrap class},
  journal={M\"unster J. Math.},
  volume={2},
  date={2009},
  pages={215--252},
  issn={1867-5778},
  review={\MRref {2545613}{}},
}

\bib{raeburn}{book}{
    AUTHOR = {Raeburn, Iain},
     TITLE = {Graph algebras},
    SERIES = {CBMS Regional Conference Series in Mathematics},
    VOLUME = {103},
 PUBLISHER = {Published for the Conference Board of the Mathematical
              Sciences, Washington, DC},
      YEAR = {2005},
     PAGES = {vi+113},
      ISBN = {0-8218-3660-9},
}


\bib{restorff}{article}{
    AUTHOR = {Restorff, Gunnar},
     TITLE = {Classification of {C}untz-{K}rieger algebras up to stable
              isomorphism},
   JOURNAL = {J. Reine Angew. Math.},
    VOLUME = {598},
      YEAR = {2006},
     PAGES = {185--210},
      ISSN = {0075-4102},
       DOI = {10.1515/CRELLE.2006.074},
       URL = {http://dx.doi.org/10.1515/CRELLE.2006.074},
}

\bib{Rordam:CK}{article}{
   author={R{\o}rdam, Mikael},
   title={Classification of Cuntz--Krie\-ger algebras},
   journal={$\K$\nb-Theory},
   volume={9},
   date={1995},
   number={1},
   pages={31--58},
   issn={0920-3036},
   doi={10.1007/BF00965458},
}


\bib{tomforde}{article}{
    AUTHOR = {Tomforde, Mark},
     TITLE = {The ordered {$K_0$}-group of a graph {$C^*$}-algebra},
   JOURNAL = {C. R. Math. Acad. Sci. Soc. R. Can.},
    VOLUME = {25},
      YEAR = {2003},
    NUMBER = {1},
     PAGES = {19--25},
      ISSN = {0706-1994},
}

   \end{biblist}
\end{bibsection}
\end{document}